\documentclass[12pt,reqno]{article}

\usepackage{amsthm,amsfonts,amssymb,amsmath,epsf, verbatim}

\usepackage{fancyhdr}
\newtheorem{theorem}{Theorem}
\newtheorem{lemma}[theorem]{Lemma}
\newtheorem{corollary}[theorem]{Corollary}

\begin{document}

\title{Upper bounds on the second largest prime factor of an odd perfect number}

\author{Joshua Zelinsky}
\date{}

\maketitle
\vspace{-1 cm}

\begin{center}
Iowa State University\\

Email:joshuaz1@iastate.edu, zelinsky@gmail.com
\end{center}

\begin{abstract} Acquaah and Konyagin showed that if $N$ is an odd perfect number with prime factorization $N= p_1^{a_1}p_2^{a_2} \cdots p_k^{a_k}$ where $p_1 < p_2 \cdots < p_k$, then one must have $p_k < 3^{1/3}N^{1/3}$. Using methods similar to theirs, we show that $p_{k-1}< (2N)^{1/5}$ and that $p_{k-1}p_k < 6^{1/4}N^{1/2}.$ We also show that if $p_k$ and $p_{k-1}$ are close to each other then these bounds can be further strengthened. 
    
\end{abstract}

Throughout this paper we will assume $N$ is an odd perfect number with $N= p_1^{a_1}p_2^{a_2} \cdots p_k^{a_k}$ where $p_1 < p_2 \cdots < p_k$ are all prime.  Acquaah and Konyagin \cite{AK} showed that one must have 
\begin{equation}
\label{p_k AK bound} p_k < 3^{1/3}N^{1/3}.
\end{equation}

We recall Euler's result that if $N$ is an odd perfect number we may write $N=q^eM^2$ where $q$ is prime, $q \equiv e \equiv 1$ (mod 4), and $(q,M)=1$. Equivalently, Euler's result states that all the $p_k$ are raised to an even power, except for a single prime $p_i$ where $p_i \equiv a_i \equiv i$ (mod 4). We will refer to this single prime raised to a $1$ (mod 4) power as the special prime.\footnote{Some authors call $q$ the ``Euler prime.'' A better name than the Euler prime in fact would be the Cartesian prime since prior to Euler's result Descartes proved that an odd perfect number needed to have exactly one prime factor raised to an odd power. In any event, the term  special prime avoids any issues of priority.}

Acquaah and Konyagin proved their result by showing that for any prime $p_i$ which is not the special prime one has $p_i < 2^{1/4}N^{1/4}$, and most of their work focuses on the possibility that $p_k$ is the special prime. Our situation is similar, although we will need to examine both the situation where  $p_k$ is the special prime as well as the situation where $p_{k-1}$ is the special prime.\footnote{Starni's recent paper \cite{Starni} claims a  very tight upper bound on the size of the special prime in general, but the author was unable to follow the proof of Corollary 2.2 in his paper which gives that result. Therefore, none of the results in this paper rely on that one.} 

It also follows immediately from their results that one has that 

\begin{equation}
\label{p_k-1 AK bound} p_{k-1} < 2^{1/4}N^{1/4}.
\end{equation}

We are interested in proving similar bounds about the second largest prime factor, $p_{k-1}$. Note that a lower bound of $p_{k-1} > 10^4$ is due to Iannucci \cite{Iannucci}. 

 In this article we prove 
\begin{equation} 
\label{New p_k-1 bound} p_{k-1}< (2N)^{1/5},     
\end{equation}
and 
\begin{equation}
\label{p_kp_k-1 bound} p_{k-1}p_k < 6^{1/4}N^{1/2}.     
\end{equation}

Note that Inequality \ref{p_kp_k-1 bound} is not a direct consequence of Inequality \ref{p_k AK bound} and Inequality \ref{New p_k-1 bound} since those two equations together would just yield $$ p_{k-1}p_k < 2^{1/5}3^{1/3}N^{\frac{8}{15}}.$$ It is also interesting to compare Inequality \ref{p_kp_k-1 bound} to the bound of Luca and Pomerance  \cite{LucaPomerance}  who proved that

\begin{equation}\label{Pom Luca bound}p_1p_2p_3 \cdots p_{k-1}p_k < 2N^{\frac{17}{26}}. \end{equation}

Of course, $\frac{1}{2}$ is less than $\frac{17}{26}$, but the left hand of Inequality \ref{p_kp_k-1 bound} only has $p_{k-1}p_k$ as opposed to the product of all the primes dividing $N$ which appears in Inequality \ref{Pom Luca bound}.\\

 As with Acquaah and Konyagin's result, the fact that for any prime $p$, $\sigma(p^a)$ and $p^a$ must be relatively prime, will play a critical role in our results

Before proving our main theorem we need the following lemma:
\begin{lemma}\label{p|q+1 and q|p^2+p+1 lemma} If $p$ and $q$ are positive odd integers such that $q \mid (p^2 + p + 1)$, and $p \mid (q + 1)$, then we must have $(p, q) = (1, 1)$ or $(p, q) = (1, 3)$.
\end{lemma}
\begin{proof} Assume that $p$ and $q$ are positive odd integers $q$ such that $q|p^2+p+1$ and $p|q+1$. So there is an $m$ such that $pm=q+1$ and we may then write $q=pm-1$. By assumption, We have  
$$pm-1|p^2+p+1$$ and hence 
$$pm-1|p^2 + p+1 + pm-1 = p(p+m+1).$$
Since $(pm-1,p)=1$ we have then that $$pm-1|p+m+1.$$ Note that $pm-1=q$ is odd, and hence $m$ is even. Thus $p+m+1$ is even, and so $2|p+m+1$. We have $$2(pm-1)|p+m+1.$$ Hence, $2pm-2 \leq p+m+1$ and 
so $$p \leq \frac{m+3}{2m-1} \leq  4.$$ From the last inequality, we must have $p=1$ or $p=3$. If $p=1$,  then $q|1^2+1+1=3$, and so $q=1$ or $q=3$. If $p=3$, then $q|3^2+3+1$, and hence $q=1$ or $q=13$, neither of which satisfies $p|q+1$. 

\end{proof}

We will first prove Inequality \ref{p_kp_k-1 bound} and then prove Inequality \ref{New p_k-1 bound}. 

\begin{theorem} We have $$p_{k-1}p_k < 2^{1/4}3^{1/4}N^{1/2}.$$
\end{theorem}
\begin{proof}

We will split our proof into two cases: Case I is where $p_k$ is the special prime; Case II is where $p_k$ is not the special prime. 

We first consider Case I where $p_k$ is the special prime. Note that if $a_k>1$, then one must have $a_k \geq 5$ and hence  $p_k^5 < N$. From this we have that $p_{k-1}<N^{1/5}$, and hence $$p_{k-1}p_k < p_k^2< N^{2/5} < 2^{1/4}3^{1/4}N^{1/2}.$$ Thus, we may assume that $a_k=1$. Since $p_k$ is the special prime, we must have that $a_{k-1}$  is even and that $a_{k-1} \geq 2$. If $a_{k-1} \geq 4$, since $(\sigma(p_{k-1}^{a_{k-1}}),p_{k-1}^{a_{k-1}})=1$ we conclude $$p_{k-1}^8< p_{k-1}^4\sigma(p_{k-1}^{4}) < 2N,$$ and hence that $p_{k-1} < 2^{1/8}N^{1/8}$ which combining with Inequality \ref{p_k AK bound}  would yield that 
$$p_{k-1}p_k < 2^{1/8}3^{1/3}N^{\frac{11}{24}} < 2^{1/4}3^{1/4}N^{1/2}.$$
Hence, we may assume that $a_{k-1} = 2 $. Then by Lemma \ref{p|q+1 and q|p^2+p+1 lemma}, we have either $p_{k-1} \not | (p_k+1)$ or $p_k \not|(p_{k-1}^2 + p_{k-1} + 1) $. We shall label the first of these two Case Ia, and the second Case Ib.

In Case Ia, since $p_{k-1} \not | (p_k+1)$ we have 
$$p_{k-1}^2(p_k+1)\sigma(p_{k-1}^2)|2N, $$ and hence
$$p_{k-1}^4p_k < p_{k-1}^2(p_k+1)\sigma(p_{k-1}^2) \leq 2N.$$

We have then $p_{k-1}^4p_k<2N$. If we then set $p_k = N^{\alpha}$ for some real number $\alpha$ we have that $p_{k-1}^4 < 2N^{1-\alpha}$, and hence $$p_{k-1}< 2^{\frac{1}{4}}N^{\frac{1-\alpha}{4}}.$$
We have then that 

$$p_{k-1}p_k < 2^{\frac{1}{4}}N^{\frac{1-\alpha}{4}}N^{\alpha} = 2^{1/4}N^{\frac{1}{4} + \frac{3}{4}\alpha}. $$

The far right-hand side of the above equation is increasing in $\alpha$, and hence largest when $\alpha$ as large as possible, namely when we have $N^\alpha= 3^{1/3}N^{1/3}$, in which case we obtain that 
$p_{k-1}p_k < 2^{1/4}3^{1/4}N^{1/2}$. 

In Case Ib, we have that $p_k \not | (p_{k-1}^2 + p_{k-1} + 1)$, and hence 
$$\sigma({p_{k-1}}^2){p_k}{p_{k-1}}^2|2N$$ and hence $$p_{k-1}^4p_k < 2N$$ and we then continue just as in Case Ia. \\  
We now consider Case II, where $p_k$ is not the special prime. Then following logic similar to that in \cite{AK}, since $(\sigma(p_k^{a_k}),p_k^{a_k})=1$ one has $$\sigma(p_k^{a_k})p_k^{a_k} | 2N,$$ and hence $p_k^4 < 2N$, and so $p_k < 2^{1/4}N^{1/4}$. Since $p_{k-1} < p_k$, we have that $$p_{k-1}p_k < (2^{1/4}N^{1/4})^2 = 2^{1/2}N^{1/2} <  2^{1/4}3^{1/4}N^{1/2}.$$
\end{proof}

We now prove

\begin{theorem} $$p_{k-1}< (2N)^{1/5}.$$
\end{theorem}
\begin{proof} Our method of proof is very similar to our method of proof above, but with a slightly different case breakdown. Case I will be when $p_{k-1}$ is the special prime. In Case II, when $p_{k-1}$ is not special, we will break this down into two cases. Case IIa is when $p_k$ is special and Case IIb is  when neither $p_k$ nor $p_{k-1}$ is the special prime. \\

Let us consider Case I, where $p_{k-1}$ is special.  Since $p_k$ is not special, $a_k$ must be even. If $a_k \geq 4$ then since we have that $\sigma(p_k^{a_k})p_k^{a_k}|2N$, this gives us $$p_k^8 < \sigma(p_k^{a_k})p_k^{a_k} \leq 2N.$$ We have then that
$p_k<2^{1/8}N^{1/8},$ and so the desired bound on $p_{k-1}$ follows.  We may thus assume that $a_k=2$. We again use Lemma \ref{p|q+1 and q|p^2+p+1 lemma} to conclude that either $p_{k-1} \not| (p_k^2 +p_k +1)$ or 
$p_k \not | (p_{k-1}+1)$. In the case that $p_k \not| p_{k-1}+1$,  we have that 
$$p_{k-1}^5 < (p_{k-1}+1)\sigma(p_k^2)p_{k-1}p_k|\sigma(2N), $$ from which the desired inequality follows. 
Similarly, in the case that $p_{k-1}$ does not divide $p_k^2 +p_k +1$, then we have that 
$p_{k-1}^5 < \sigma(p_k^2)p_{k-1}p_k^2|2N $. \\

Case IIa  where $p_k$ is the special prime is nearly identical to Case I.\\

Now consider Case IIb, where neither $p_k$ nor $p_{k-1}$ is special. By logic similar to that in Case I, we may assume that $$a_{k-1}=a_k=2.$$ We note that $$p_k^2 > p_{k-1}^2 + p_{k-1} +1, $$ and hence $p_k^2 \not | p_{k-1}^2 + p_{k-1} + 1$. We have 
$$p_k\sigma(p_{k-1}^2)p_{k-1}^2 |2N,$$ and so we have that $$p_{k-1}^5 < p_k\sigma(p_{k-1}^2)p_{k-1}^2 \leq 2N $$ from which the desired inequality follows. 
 
\end{proof}

Based on the ease of proving stronger inequalities for $p_{k-1}$ than for $p_k$ it seems natural to ask if we can prove that $p_k$ must be much larger than $p_{k-1}$. Right now, it is not obvious how to even rule out a situation like $p_k=p_{k-1}+2$. However we can prove the following weak result:

\begin{theorem} If $p_k < p_{k-1} + \sqrt{3p_{k-1}} -2,$ then we have $$p_{k-1} < 2^{\frac{1}{6}}N^{\frac{1}{6}}.$$
\label{p_k and p_k-1 close result}
\end{theorem}

To prove this theorem we need the following lemma:

\begin{lemma} Let $p$ be an odd prime and assume that $q>3$ is a prime such that $q|(p^2+p+1)$ then one has $|p-q|> \sqrt{3p}-2.$
\label{p^2 + p +1 does not have any divisors near its square root}
\end{lemma}
\begin{proof} Assume that $p$ and $q$ are primes such that $q|p^2+p+1$. The case of $p=3$ is easy to check so assume that $p>3$. Set $p=q+k$, so $q=p-k$.  We note then that $$k^2+k+1= p^2+p+1 - (p+k+1)(p-k),$$ and hence $p-k|k^2+k+1.$ If $p-k=k^2+k+1$ then we have $p=k^2+2k+1=(k+1)^2$. But $p$ is prime and so cannot be a perfect square.  We must then have $m(p-k)=k^2+k+1$. We note that $m$ is odd and greater than 1. So we have $m \geq 3$. We thus have
$$3(p-k) \leq k^2+k+1. $$ 
and hence $3p \leq k^2 +4k + 1  < (k+1)^2$. Therefore, $$\sqrt{3p} < |k+1| $$ from which the desired bound follows. 
\end{proof}

We are now in the position to prove Theorem \ref{p_k and p_k-1 close result}.
\begin{proof}  Assume that $p_k < p_{k-1} + \sqrt{3p_{k-1}} -1.$ We wish to show that $p_{k-1} < 2^{\frac{1}{6}}N^{\frac{1}{6}}.$ Note that this result follows easily if we have either $p_k^4|N$ or have $p_{k-1}^4|N$. Thus, we may assume that we have $a_k$ and $a_{k-1}$ are either 1 or 2. Let us consider the case where $a_k=a_{k-1}=2$. We have then by Lemma \ref{p^2 + p +1 does not have any divisors near its square root} that $$p_{k-1} \not|p_k^2+p_k +1=\sigma(p_k^2)$$ and we get the same for swapping $p_k$ and $p_{k-1}$. Hence we have that
$$(p_k^2+p_k+1)(p_{k-1}^2+p_{k-1}+1)p_{k-1}^2p_k^2 |2N.$$ We thus have $p_{k-1}^8 <  2N$, and hence
$$p_{k-1} < 2^{\frac{1}{8}}N^{\frac{1}{8}} < 2^{\frac{1}{6}}N^{\frac{1}{6}}. $$

We thus may assume that one of $p_k$ and $p_{k-1}$ is special and the other is raised to the second power. We will look at the case where $p_k$ is special (the other case is nearly identical). We note that we cannot have $p_{k-1}|p_k+1$ since this would force $p_{k-1} \leq \frac{p_k+1}{2}$ which would contradict our assumption.

From the above note and Lemma \ref{p^2 + p +1 does not have any divisors near its square root}, we have  $$(p_{k-1}^2+p_{k-1}+1)p_{k-1}^2p_k(p_k+1)|2N,$$ and so 
$$p_{k-1}^6 < (p_{k-1}^2+p_{k-1}+1)p_{k-1}^2p_k(p_k+1) \leq 2N.$$  From this last inequality our desired inequality immediately follows. 

\end{proof}

We also have the following assertion as an easy corollary of Theorem \ref{p_k and p_k-1 close result}

\begin{corollary} If  $p_k < p_{k-1} + \sqrt{3p_{k-1}} -2,$ then $$p_1p_2p_3 \cdots p_k < 3N^{\frac{7}{12}}, $$ and
$$p_{k-1}p_k < 2N^{\frac{1}{3}}.  $$\end{corollary}

Note that $\frac{7}{12} < \frac{17}{26}$, so the inequality for the radical in this corollary is stronger than the Luca and Pomerance bound albeit it requires the additional hypothesis that $p_k$ and $p_{k-1}$ are close.  \\

There are a variety of avenues for future work. We are optimistic that careful refinements of these sorts of arguments can further improve bounds on $p_k$ and $p_{k-1}$. We also suspect that recent results showing that an odd perfect number must have many repeated prime factors such as \cite{OchemRao1} and \cite{Zelinsky} may be used to improve either these results or the results of Acquaah and Konyagin for $p_k$. The primary barrier to doing so appears to be that one may have an odd perfect number with a prime $p$ and many distinct $p_i$ such that $\sigma(p_i^{a_i})$ is a power of $p$. This situation corresponds to systems of extremely restrictive Diophantine equations, so limiting how many such $p_i$ one can have may be possible.  It also may be possible to combine these sorts of bounds with the sort of estimates of the index of an odd perfect number as introduced by Dris\cite{Dris}\cite{DrisLuca}. \par 

One other potential direction to go in is to obtain non-trivial estimates for $p_{k-2}$, or more generally to give a bound on $p_{k-i}$ for $1\leq i\leq k$. One easily has from Euler's theorem about an odd perfect number that $$p_{k-i} < N^{\frac{1}{2i+1}}.$$ A bound for $p_{k-2}$ better than $N^{\frac{1}{5}}$ then is a natural goal since $N^{1/5}$ is the bound one gets immediately from applying Euler's theorem. 

{\bf Acknowledgements} The author would like to acknowledge the helpful comments from the referee. The author would also like to acknowledge Douglas McNeil  who noted that a prior version of this paper had an incorrect version of Lemma \ref{p^2 + p +1 does not have any divisors near its square root}. Helpful comments were also provided by Carl Pomerance.

\end{document}